\documentclass[12pt]{article}
\usepackage{amsmath}
\usepackage{amsthm}
\usepackage{amssymb}
\usepackage{latexsym}
\usepackage[latin1]{inputenc}
\usepackage{amsfonts}
\usepackage{multirow}
\usepackage{hhline}
\usepackage{graphics}
\usepackage{graphicx}
\usepackage{color}
\usepackage{harvard}
\usepackage[latin1]{inputenc}
\begin{document}
\newtheorem{proposition}{Proposition}
\def\E{\operatorname{\bf E}}
\def\P{\operatorname{\bf P}}
\def\1{{\bf 1}}
\def\C{{\mathbb C}}
\renewcommand{\Im}{\operatorname{\rm Im}}
\renewcommand{\Re}{\operatorname{\rm Re}}
\def\process#1{#1=\{#1_t\}_{t\geq 0}}
\newtheorem{theorem}{Theorem}[section]
\newtheorem{remark}{Remark}[section]
\newtheorem{example}{Example}[section]
\newtheorem{lem}{Lemma}[section]
\newtheorem{assumption}{Assumption}
\newtheorem{corollary}{Corollary}[section]
\newtheorem{propo}{Proposition}[section]
\newtheorem{conj}{Conjecture}
\newtheorem{defi}{Definition}
\def\la{\lambda}
\def\beq{\begin{equation}}
\def\eeq{\end{equation}}
\newcommand{\R}{I \! \! R}

\title{
Symmetry and Time Changed Brownian Motions\thanks{J. Fajardo
Thanks CNPq for financial support}}
\author{Jos\'{e} Fajardo\thanks{IBMEC Business School, Rio de
Janeiro - Brazil, e-mail: pepe@ibmecrj.br}\hspace{0.2cm}  and
 Ernesto Mordecki\thanks
{Centro de Matem\'atica, Facultad de Ciencias, Universidad de la
Rep\'ublica, Montevideo. Uruguay. e-mail: mordecki@cmat.edu.uy.}}
\date{\today}
\maketitle
\begin{abstract}
In this paper we examine which Brownian Subordination with drift
exhibits the symmetry property introduced by
\citeasnoun{FajardoMordecki2006b}. We  obtain that when the
subordination results in a Lévy process, a necessary and
sufficient condition for the symmetry to hold is that drift must
be equal to -1/2.
\end{abstract}

{\bf  Key Words:} {Time Changed, Subordination, Symmetry.}\\

\section{Introduction}
\hspace{0.50cm}Asset returns have been studied during many years,
one of the main findings is the presence of many small jumps in a
finite time interval. To deal with that fact more realistic jump
structures have been suggested, as for example the Generalized Hyperbolic (GH) model of \citeasnoun{EberleinKellerPrause98}, the Variance-Gamma (VG) model of \citeasnoun{MadanCarrChang98} and the CGMY model of \citeasnoun{CGMY02}. All that models are included in a huge family called L\'{e}vy processes.\\

By a result due to \citeasnoun{monroe}, we know that every
semimartingale can be written as a time-changed Brownian motion.
As a consequence many Lévy processes can be represented as
Time-changed Brownian Motion. This fact is very useful for the
pricing of multiasset derivatives, since we can correlate assets
by correlating the Brownian motions. \\

On the other hand, it is not easy to express explicitly the
time-change used. As for example in the case of CGMY process
introduced by \citeasnoun{CGMY02} and the Meixner process,
introduced by \citeasnoun{Grigelionis99} and
\citeasnoun{Schoutens2002}. The time-changes used for these
processes have been obtained recently by \citeasnoun{madanyor06}.\\

In this paper, based on this explicit time-changes, we study how
does the symmetry concept, introduced by
\citeasnoun{FajardoMordecki2006b}, works in the CGMY and Meixner
models and also in other Lévy processes obtained by Subordinating
Brownian motion with
drift.\\

The paper is organized as follows: in Section 2, we introduce
Time-changed Brownian Motion. In Section 3, we describe the market
model. In Section
 4, we describe symmetry
and obtain our main result. In last sections we have the
conclusions and an appendix.

\section{Time-Changed Brownian Motion}

\hspace{0.50cm} Let ${X}=(X^1,\dots,X^d)$ be a $d$-dimensional
L\'evy process respect to the complete filtration
$\mathbf{F}=\{\mathcal{F}_t, t\geq 0\}$, this process is defined
on the probability space $(\Omega, \mathbf{F},P)$, in other words
$X$ is a càdlàg process with independent and
stationary increments.\\

We know by the Lévy-Khintchine formula that the characteristic
function of $X_t$, $\phi_{X_t}(z)\equiv {\sf
E}e^{izX_t}=\exp(t\psi(z))$ where the characteristic exponent
$\psi$ is given by:

\beq \label{e:charexp} \psi(z)=i(b,z)-\frac 12 (z,\Sigma
z)+\int_{\R^d}\Big(e^{i(z,y)}-1-i(z,y){\bf 1}_{\{|y|\le
1\}}\Big)\Pi(dy), \eeq where $b=(b_1,\dots,b_d)$ is a vector in
$\R^d$, $\Pi$ is a positive measure defined on
$\R^d\setminus\{0\}$ such that $\int_{\R^d}(|y|^2\wedge 1)\Pi(dy)$
is finite, and $\Sigma=((s_{ij}))$ is a symmetric nonnegative
definite matrix, that can always be written as $\Sigma=A'A$ (where
$'$ denotes transposition) for some matrix $A$.\\

Now let $t \mapsto \mathcal{T}_t,\;\;, t\geq 0,$ be an increasing
cádlág process, such that for each fixed $t$, $\mathcal{T}_t$ is a
stopping time with respect to $\mathbf{F}$. Furthermore, suppose
$\mathcal{T}_t$ is finite and positive $P-a.s.,\;\forall t\geq 0$
and $\mathcal{T}_t\rightarrow\infty$ as $t\rightarrow\infty$. Then
$\{\mathcal{T}_t\}$ defines a random change on time, we can also
impose
$E\mathcal{T}_t=t$.\\

Now let $X_t=W_t$ be a Brownian motion. Then, consider the process
$Y_t$ defined by:
$$Y_t\equiv X_{\mathcal{T}_t},\;\;t\geq 0,$$
this process is called Time-changed Brownian Motion. Using
different time changes $\mathcal{T}_t$, we can obtain a good
candidate for the underlying asset return process. We know that if
$\mathcal{T}_t$ is a Lévy process we have that $Y$ would be
another Lévy process\footnote{See (\citeasnoun[Th. 4.2 pag.
108]{CT04})}. A more general situation is when $\mathcal{T}_t$ is
modelled by a non-decreasing semimartingale:
\begin{equation} \mathcal{T}_t=a_t+\int_0^t\int_0^\infty y\mu
(dy,ds)\label{carrwu}
\end{equation}
 where $a$ is a drift and $\mu$ is the counting measure of
jumps of the time change. Now we can obtain the characteristic
function of $Y_t$:
$$\phi_{Y_t}(z)=
{\sf E}(e^{iz'X_{\mathcal{T}_t}})={\sf E}\left({\sf E}\left
(e^{iz'X_u}\mid T_t=u\right)\right)$$ If $\mathcal{T}_t$ and $X_t$ are
independent, then:

\begin{equation}
\phi_{Y_t}(z)=\mathcal{L}_{\mathcal{T}_t}(\psi(z)),\label{char}
\end{equation}
where $\mathcal{L}_{\mathcal{T}_t}$ is the Laplace transform of
$\mathcal{T}_t$. So if the Laplace transform of $\mathcal{T}$ and
the characteristic exponent of $X$ have closed forms, we can
obtain a closed form for $\phi_{Y_t}$, as we show in the next
examples.\\

In this way we can obtain the distribution of $Y_t$ for every $t$
and in this way we can price some derivatives.

\subsection{Subordinators}
We say that  ${\mathcal{T}_t}$ is a
{\it Subordinator}
if it is a Lévy Processes with non decreasing trajectories.
As a consequence, trajectories take positive values almost sure.
These properties are necessary in order to make a time change,
a desired fact for a Time-changed and the choose of a Lévy
process will allow us to obtain as a result a very good candidate
to model asset returns.\\

\subsubsection{Stable subortination} Now let ${\mathcal{T}_t}$ be a
$\alpha-$Stable with zero drift and $\alpha\in (0,1)$, that is a
Lévy process with Lévy measure given by:

$$
\rho(x)= \frac {A}{x^{1+\alpha}},\;\;x>0,$$ we can compute the
Laplace transform of ${\mathcal{T}_t}$ :
$$
\mathcal{L}_{\mathcal{T}_t}(z)=A\int_0^\infty
\frac{e^{zx}-1}{x^{\alpha+1}}dx=-\frac
{A\Gamma(1-\alpha)}{\alpha}(-z)^\alpha,
$$
Let $X_t$ be a symmetric $\beta-$stable process,
that is using eq. (\ref{e:charexp}) we have
$$\psi(z)=-B|z|^\beta,$$
where $A$ and $B$ are positive constants. Then Using eq.
(\ref{char}), we have that $Y_t=X_{\mathcal{T}_t}$ has
characteristic exponent given by

$$
\phi_{Y_t}(z)=\mathcal{L}_{\mathcal{T}_t}(\psi(z))=-C|z|^{\beta\alpha}
$$
where $C=\frac {AB^\alpha\Gamma(1-\alpha)}{\alpha}$. That is $Y_t$
is a $\beta\alpha-$Stable symmetric process. If $X_t$ be a
Brownian Motion, i.e. $\beta=2$, then $Y_t$ would be a
$2\alpha-$Stable symmetric process. As $\alpha <1$, we have that
$Y_t$ will be a process with heavy tails, which is an stylized
fact of the majority of the observed asset returns.
\subsubsection{Tempered subordination}
Assume that $\mathcal{T}_t$ has Lévy measure given by
$$\rho(x)=\frac {Ce^{-\lambda x}}{x^{1+\alpha}}1_{x>0},$$
where $C,\lambda>0$ and $\alpha\in (0,1).$ Then, we have
$$\mathcal{L}_{\mathcal{T}_t}(z)=
C\Gamma(-\alpha)\left[(\lambda-z)^\alpha-\lambda^\alpha\right]
$$
Now let $X_t$ be a Time-Changed Brownian motion with drift $\mu$,
i.e. $X_t=\mathcal{T}_t \mu +\sigma W(\mathcal{T}_t)$.
Then, using eq. (\ref{e:charexp}), we have:

$$
\psi(z)=-\frac {z^2}{2}\sigma^2+i\mu z,
$$
and using eq. (\ref{char}), we have that
$Y_t=X_{\mathcal{T}_t}$
has characteristic exponent given by

$$
\phi_{Y_t}(z)
=
\mathcal{L}_{\mathcal{T}_t}(\psi(z))
=
C
\Gamma(-\alpha)
\left[
(\lambda+\frac {z^2}{2}\sigma^2-i\mu z)^\alpha-\lambda^\alpha
\right]
$$

\section{Market Model}

Consider a Time-changed Brownian market where we have a riskless asset,
with price process denoted by
$\process B$,
with
\begin{equation*}
B_t=e^{rt},\qquad r\ge 0,
\end{equation*}
where we take $B_0=1$ for simplicity, and a risky asset, with
price process denoted by $\process S$,
\begin{equation}\label{e:market}
S_t=S_0e^{Y_t},\qquad S_0=e^y>0.
\end{equation}
Where $Y_t$ is a time-changed Brownian Motion with a drift, where
the time change is an independent Subordinator. Denote by
$(b,\sigma,\nu)$ the characteristics of the time-changed Brownian
with drift process\footnote{Here we assume conditions to guarantee
that this process is a Lévy process, see Appendix.}. Also, we
assume that the stock pays dividends, with constant rate
$\delta\ge 0$, and we assume that the probability measure $P$ is
the chosen equivalent martingale measure. In other words, prices
are computed as expectations with respect to $P$, and the
discounted and reinvested process
$\{e^{-(r-\delta)t}S_t\}$ is a $P$--martingale.\\

In order to this condition be satisfied, we need that
$$E\left[e^{-(r-\delta)t}S_t\right]=S_0,\forall t$$
In other words, $E(e^{Y_t})=e^{(r-\delta)t}$. That means that the
characteristic exponent of $Y$ must satisfy:
$$\psi_Y(1)=(r-\delta).$$ So to avoid arbitrage opportunities we
have to restrict our attention to time-changed Brownian process
such
that the exponential process $e^{Y_t-(r-\delta)t}$
be a $P$--martingale.\\

\section{Symmetry}
Consider a Time-changed Brownian market described above with
driving process characterized by $(b,\sigma,\nu)$. Now, consider a
market model with two assets, a deterministic savings account
$\process{\tilde B}$, given by
\begin{equation*}
\tilde B_t=e^{\delta t},\qquad r\ge 0,
\end{equation*}
and a stock $\process{\tilde S}$, modelled by
\begin{equation*}
\tilde{S}_t=Ke^{\tilde{Y}_t},\qquad S_0=e^x>0,
\end{equation*}
where $\process{\tilde{Y}}$ is a Lévy processes with
characteristics under $\tilde{P}$ given by
$(\tilde{b},\tilde{\sigma},\tilde{\nu})$. This market is the
\emph{dual market} in \citeasnoun{FajardoMordecki2006b}. Observe,
that in the dual market (i.e. with respect to $\tilde{P}$), the
process $\{e^{-(\delta-r)t}{\tilde
S}_t\}$ is a martingale.\\

 It is interesting to notice, that in a market with no jumps the
distribution (or laws) of the discounted (and reinvested) stocks
in both the given and dual markets coincide. It is then natural to
define a market to be \emph{symmetric} when this relation hold,
i.e. when
\begin{equation}\label{e:symmetry}
{\cal L}\big(e^{-(r-\delta)t+Y_t}\mid P\big)={\cal
L}\big(e^{-(\delta-r)t-Y_t}\mid\tilde{P}\big),
\end{equation}
meaning equality in law. \citeasnoun{FajardoMordecki2006b} derived
the characteristics of the {\it dual process} $\widetilde{Y}_t$,
In particular they obtained that a necessary and sufficient
condition for \eqref{e:symmetry} to hold is
\begin{equation}\label{e:pisymmetric}
\nu(dx)=e^{-x}\nu(-dx).
\end{equation}
This ensures $\tilde{\nu}=\nu$, and from this follows
$b-(r-\delta)=\tilde{b}-(\delta-r)$, giving \eqref{e:symmetry}, as
we always have $\tilde{\sigma}=\sigma$. \\

Now as we have assumed that
$Y_t=\mathcal{T}_t\mu+W(\mathcal{T}_t)$. If we denote the Lévy
measure of $\mathcal{T}_t$ by $\rho(dy)$. Then, we know by
\citeasnoun{Sato99}[Th. 30.1] that Lévy measures are related by
\begin{equation*}
\nu(dx)=\left[ \int_0^\infty \frac {1}{\sqrt{2\pi y}} e^{-\frac
{(x-\mu y)^2}{2y}}\rho(dy)\right]dx,
\end{equation*}
 we can express this relationship as
\begin{equation*}
\nu(dx)=e^{\mu x}f(x)dx,
\end{equation*}
where $$f(x)=\int_0^\infty \frac {1}{\sqrt{2\pi y}}
e^{-\frac{1}{2y}(x^2+\mu^2 y^2)}\rho(dy),$$ this density is even,
i.e. $f(x)=f(-x)$.
\begin{proposition}
A {\it Time-Changed Brownian Market} is symmetric if and only if
the drift is equal -1/2.
\end{proposition}

\begin{proof}
By condition (\ref{e:pisymmetric}) we have that a market is
symmetric iff
$$e^{\mu x}f(x)dx=e^{-x}[e^{-\mu x}f(-x)],$$
from here $\mu=-1/2$.
\end{proof}

As an application of this Proposition we have the following
\begin{corollary}
\begin{itemize}
\item[a)] The CGMY Market model will be symmetric if and only if
$$G-M=-1.$$
\item[b)]The Meixner Market Model will be symmetric if and only if
$$2b+a=0$$
\end{itemize}
\end{corollary}
\begin{proof}
 Since, \citeasnoun{madanyor06} have obtained the
explicitly representations of CGMY model and Meixner Model as a
Time-changed Brownian motion with drift $\frac {G-M}{2}$ and
$\frac ba$, respectively. The result follows.
\end{proof}
Moreover, in \citeasnoun{CGMY02} the drift is estimated under the
market risk neutral measure and in the majority of cases the drift
is negative and less that -0.5. Also, \citeasnoun{Scoutens2001}
estimates the values of $a$ and $b$ and obtain similar evidence.
It give us evidence against market symmetry.
\section{Conclusions} \hspace{0.5cm} In a Time-changed Brownian market
we have shown that market will be symmetric if and only if the drift is equal to -1/2.\\

 Since, Time-Changed Brownian motion allow us to model
correlations by correlating the Brownian motions. Another
important application that can be address with duality techniques
is the pricing of bidimensional derivatives in a Time-Changed
Brownian context as is done by \citeasnoun{FajardoMordecki2006a}
for the case of Lévy processes.
\section{Appendix}
The following Theorem is taken from \citeasnoun{CT04}[Th. 4.3,
Pag. 113] and it gives conditions for a Time-Changed Brownian
motion with drift to be a Lévy process.
\begin{theorem}
Let $\nu$ be a Levy measure and $\mu\in \R$. There exists a Lévy
process $Y_t$ with Lévy measure $\nu$ such that $Y_t = W(Z_t) +
\mu Z_t$ for some subordinator $Z_t$ and some Brownian motion
$W_t$ independent from $Z$, if and only if the following
conditions are satisfied:
\begin{enumerate}
\item $\nu$ is absolutely continuous with density $\nu(x)$
 \item $\nu (x)e^{-\mu x}= \nu(-x)e^{\mu x}$ \item
$\nu(\sqrt(x))e^{-\mu\sqrt(x)}$ is a completely monotonic function
on $(0,1)$
\end{enumerate}
\end{theorem}

A function $f:[a,b]\rightarrow \R $ is called completely monotonic
if all derivatives exist and $(-1)k \frac {d^k
f(u)}{du^k}>0,\;\;\forall k \geq 1.$

\bibliographystyle{econometrica}
\bibliography{catalog}
\end{document}